\documentclass[reqno,12pt]{amsart}
\theoremstyle{plain}
\newtheorem{thm}{Theorem}[section]
\newtheorem{cor}[thm]{Corollary} 
\newtheorem{lemma}[thm]{Lemma} 
\newtheorem{prop}[thm]{Proposition}

\theoremstyle{remark}

\theoremstyle{definition}
\newtheorem{defi}[thm]{Definition}
\newtheorem{example}[thm]{Example}

\newtheorem{remark}[thm]{Remark}

\newtheorem{notation}[thm]{Notation}

\usepackage{amscd,amssymb,comment,epic,eepic,euscript,graphics}
\usepackage{enumerate}

\newcommand\Abar{{\overline A}}
\newcommand\Afr{{\mathfrak A}}

\newcommand\Cc{{\mathcal{C}}}

\newcommand\conv{\operatorname{conv}}
\newcommand\Cpx{{\mathbf C}}
\newcommand\diag{\text{\rm diag}}
\newcommand\eps{\epsilon}
\newcommand\Fc{{\mathcal{F}}}
\newcommand\Gc{{\mathcal{G}}}
\newcommand\id{{\operatorname{id}}}

\newcommand\KEu{{\EuScript K}}                   

\newcommand\Mcal{{\mathcal{M}}}
\newcommand\Nats{{\mathbf N}}
\newcommand\rank{\mathrm{rank}\,}

\newcommand\Reals{{\mathbf R}}

\newcommand\Tcirc{{\mathbf T}}
\newcommand\tr{{\mathrm{tr}}}
\newcommand\Uc{{\mathcal{U}}}

\newcommand\Vt{{\widetilde V}}
\newcommand\Xbar{{\overline X}}
\newcommand\Ybar{{\overline Y}}
\newcommand\Xh{{\widehat X}}

\newcommand\Yh{{\widehat Y}}

\begin{document}

\title[Matrices of unitary moments]{Matrices of unitary moments}

\author[Dykema, Juschenko]{Ken Dykema$^{*}$, Kate Juschenko}
\address{Department of Mathematics, Texas A\&M University,
College Station, TX 77843-3368, USA}
\email{kdykema@math.tamu.edu, juschenko@math.tamu.edu}
\thanks{\footnotesize $^{*}$Research supported in part by NSF grant DMS-0600814}




\begin{abstract}
We investigate certain matrices composed of mixed, second--order moments of unitaries.
The unitaries are taken from C$^*$--algebras with moments taken with respect to
traces, or, alternatively, from matrix algebras with the usual trace.
These sets are of interest in light of a theorem of E.\ Kirchberg about Connes' embedding problem.
\end{abstract}

\maketitle

\section{Introduction}

One fundamental question about operator algebras is
Connes' embedding problem, which in its original formulation asks
whether every II$_1$--factor $\Mcal$ embeds in the ultrapower
$R^\omega$ of the hyperfinite II$_1$--factor $R$.
This is well known to be equivalent to the question of whether all elements of II$_1$--factors
possess matricial microstates, (which were introduced by Voiculescu~\cite{V94} for free entropy),
namely, whether such elements are approximable in $*$--moments by matrices.
Connes' embedding problem is 
known to be equivalent to a number of different problems, in large part due to a remarkable
paper~\cite{Ki93} of Kirchberg.
(See also the survey~\cite{Oz04}, and the papers \cite{P96}, \cite{Ra99}, \cite{Ra04}, \cite{B04}, \cite{Ra06},
\cite{CD08}, \cite{KS08}, \cite{Ra}, \cite{JP} for results with bearing on Connes' embedding problem.)

In Proposition~4.6 of~\cite{Ki93},
Kirchberg proved that, in order to show that a finite von Neumann algebra $\Mcal$ with faithful tracial state
$\tau$ embeds in R$^\omega$,
it would be enough to show that for all $n$, all unitary elements $U_1,\ldots,U_n$ in $\Mcal$
and all $\eps>0$,
there is $k\in\Nats$ and there are $k\times k$ unitary matrices $V_1,\ldots,V_n$ such that $|\tau(U_i^*U_j)-\tr_k(V_i^*V_j)|<\eps$
for all $i,j\in\{1,\ldots,n\}$,
where $\tr_k$ is the normalized trace on $M_k(\Cpx)$.
(He also required $|\tau(U_i)-\tr_k(V_i)|<\eps$, but this formally stronger condition
is easily satisfied by taking the $n+1$ unitaries $U_1,\ldots,U_n,U_{n+1}=I$ in $\Mcal$
finding $k\times k$ unitaries $\Vt_1,\ldots,\Vt_{n+1}$, so that $|\tau(U_i^*U_j)-\tr_k(\Vt_i^*\Vt_j)|<\eps$,
and letting $V_i=\Vt_{n+1}^*\Vt_i$.)
It is, therefore, of interest to consider the set of possible
second--order mixed moments of unitaries in such $(\Mcal,\tau)$
or, equivalently, of unitaries in C$^*$--algebras with respect to tracial states.
(See also~\cite{Ra99}, where some similar sets were considered by F.\ R\u adulescu.)

\begin{defi}\label{def:G}
Let $\Gc_n$ be the set of all $n\times n$ matrices $X$ of the form
\begin{equation}\label{eq:X}
X=\big(\tau(U_i^*U_j)\big)_{1\le i,j\le n}
\end{equation}
as $(U_1,\ldots, U_n)$ runs over all $n$--tuples of unitaries
in all C$^*$--algebras $A$ possessing a faithful tracial
state $\tau$.
\end{defi}

\begin{remark}\label{rem:GNS}
The set--theoretic difficulties in the phrasing of Definition~\ref{def:G} can be evaded by insisting that $A$ be represented on
a given separable Hilbert space.
Alternatively,
let $\Afr=\Cpx\langle U_1,\ldots,U_n\rangle$
denote the universal, unital, complex $*$--algebra generated by unitary elements $U_1,\ldots,U_n$.
A linear functional $\phi$ on $\Afr$ is positive if $\phi(a^*a)\ge0$ for all $a\in\Afr$.
By the usual Gelfand--Naimark--Segal construction, any such positive functional $\phi$
gives rise to a Hilbert space $L^2(\Afr,\phi)$ and a $*$--representation $\pi_\phi:\Afr\to B(L^2(\Afr,\phi))$.
Thus, the set $\Gc_n$ equals the set of all matrices $X$ as in~\eqref{eq:X}
as $\tau$ runs over all positive, tracial, unital, linear functionals $\tau$ on $\Afr$.
\end{remark}

\begin{defi}\label{def:F}
Let $\Fc_n$ be the closure of the set
\[
\big\{\big(\tr_k(V_i^*V_j)\big)_{1\le i,j\le n}\mid k\in\Nats,\,V_1,\ldots,V_n\in\Uc_k\big\},
\]
where $\Uc_k$ is the group of $k\times k$ unitary matrices.
\end{defi}

A {\em correlation matrix} is a complex,
positive semidefinite
matrix having all diagonal entries equal to $1$.
Let $\Theta_n$ be the set of all $n\times n$ correlation matrices.
Clearly, we have
\[
\Fc_n\subseteq\Gc_n\subseteq\Theta_n\,.
\]
Kirchberg's result is that Connes' embedding problem is equivalent to
the problem of whether $\Fc_n=\Gc_n$ holds for all $n$.

\begin{prop}
For each $n$,
\begin{enumerate}[(i)]
\item \label{it0} $\Fc_n$ and $\Gc_n$ are invariant under conjugation with $n\times n$ diagonal unitary matrices
and permutation matrices,
\item \label{it1} $\Fc_n$ and $\Gc_n$ are compact, convex subsets of $\Theta_n$,
\item \label{it2} $\Fc_n$ and $\Gc_n$ are closed under taking Schur products of matrices.
\end{enumerate}
\end{prop}
\begin{proof}
Part~(\ref{it0}) is clear.
Note that $\Theta_n$ is a norm--bounded subset of $M_n(\Cpx)$.
That $\Fc_n$ is closed is evident.
That $\Gc_n$ is closed follows from the description in Remark~\ref{rem:GNS} and the fact
that a pointwise limit of positive traces on $\Afr$ is a positive trace.
This proves compactness.
Convexity of $\Fc_n$ follows from by observing that if $V$ is a $k\times k$ unitary and $V'$ is a $k'\times k'$ unitary,
then for arbitrary $\ell,\ell'\in\Nats$,
\[
\underset{\ell\text{ times}}{\underbrace{V\oplus\cdots\oplus V}}\oplus
\underset{\ell'\text{ times}}{\underbrace{V'\oplus\cdots\oplus V'}}
\]
can be realized as a block--diagonal $(k\ell+k'\ell')\times(k\ell+k'\ell')$ matrix whose normalized trace is
\[
\frac{k\ell}{k\ell+k'\ell'}\tr_k(V)+\frac{k'\ell'}{k\ell+k'\ell'}\tr_{k'}(V').
\]
Convexity of $\Gc_n$ follows because a convex combination of positive traces on $\Afr$ is a positive trace.
This proves~\eqref{it1}.

Closedness of $\Fc_n$ under taking Schur products follows by observing that if $V$ and $V'$ are unitaries as above,
then $V\otimes V'$ is a $kk'\times kk'$ unitary whose normalized trace is $\tr_k(V)\tr_{k'}(V')$.
For $\Gc_n$, we observe that if $U$ and respectively, $U'$, are unitaries in C$^*$--algebras $A$ and $A'$ having tracial states $\tau$
and $\tau'$, then the spatial tensor product C$^*$--algebra $A\otimes A'$ has tracial state $\tau\otimes\tau'$ that takes value
$\tau(U)\tau'(U')$ on
the unitary $U\otimes U'$.
This proves~\eqref{it2}.
\end{proof}

Since it is important to decide whether we have $\Fc_n=\Gc_n$ for all $n$, it is interesting to learn more about the sets $\Fc_n$.
A first question is whether $\Fc_n=\Theta_n$ holds.
In Section~\ref{sec:ext}, we show that this holds for $n=3$ but fails for $n\ge4$.
The proof relies on a characterization of extreme points of $\Theta_n$, and it uses also
the set $\Cc_n$ of matrices of moments of commuting unitaries.
In Section~\ref{sec:Reals} we prove $M_n(\Reals)\cap\Theta_n\subseteq\Fc_n$, and some further results
concerning $\Cc_n$.
In Section~\ref{sec:int}, we show that $\Fc_n$ has nonempty interior, as a subset of $\Theta_n$.

\section{Extreme points of $\Theta_n$ and some consequences}
\label{sec:ext}

The set $\Theta_n$ of $n\times n$ 
correlation matrices is embedded in the affine space consisting of the self--adjoint complex
matrices having all diagonal entries equal to $1$;
it is just the intersection of the set of positive, semidefinite matrices with this space.
Every element of $\Theta_n$ is bounded in norm by $n$
(\emph{cf}\/ Remark~\ref{rem:Z=0}), and $\Theta_n$ is a compact, convex space.
Since, in the space of self--adjoint matrices, every positive definite matrix is the center of a ball consisting
of positive matrices, it is clear that the boundary of $\Theta_n$ (for $n\ge2$) consists of singular matrices.

The extreme points of $\Theta_n$ and $\Theta_n\cap M_n(\Reals)$
have been studied in~\cite{CV79}, \cite{L80}, \cite{GPW90} and~\cite{LT94}.
In this section, we will use an easy characterization of the extreme points of $\Theta_n$
to draw some conclusions about matrices of unitary moments.
The papers cited above contain the facts about extreme points of $\Theta_n$ found below, and have results going
well beyond.
However, for completeness and for use later in examples, we provide proofs, which are brief.

We also introduce the subset $\Cc_n$ of $\Fc_n$, consisting of matrices 
of moments of commuting unitaries.

This is a convenient place to recall the following standard fact.
We include a proof for convenience.
\begin{lemma}\label{lem:frame}
The set of all $X\in\Theta_n$ of rank $r$ is the set of all
frame operators $X=F^*F$ of frames $F=(f_1,\ldots,f_n)$, consisting of $n$ unit vectors $f_j\in\Cpx^r$, where $r=\rank(X)$.
If, in addition, $X\in M_n(\Reals)$, then the frame $f_1,\ldots,f_n$ can be chosen in $\Reals^r$.
\end{lemma}
\begin{proof}
Every frame operator $F^*F$ as above clearly belongs to $\Theta_n$ and has rank $r$.

Recall that the support projection of a Hermitian matrix $X$
is the projection onto the orthocomplement of the nullspace of $X$.
Let $P$ be the support projection of $X$ and let $\lambda_1\ge\cdots\ge\lambda_r>0$
be the nonzero eigenvalues of $X$ with corresponding orthonormal eigenvectors $g_1,\ldots,g_r\in\Cpx^n$.
Let $V:\Cpx^r\to P(\Cpx^n)$ be the isometry defined by $e_i\mapsto g_i$, where $e_1,\ldots,e_r$ are the standard
basis vectors of $\Cpx^r$.
So $P=VV^*$.
Then $X=F^*F$, where $F$ is the $r\times n$ matrix
\[
F=V^*X^{1/2}=\diag(\lambda_1,\ldots,\lambda_r)^{1/2}V^*.
\]
If $f_1,\ldots,f_n\in\Cpx^r$ are the columns of $F$, then $\|f_i\|=X_{ii}=1$ and the linear span of
$f_1\ldots,f_n$ is $\Cpx^r$.
Thus, $f_1,\ldots, f_n$ comprise a frame.

If $X$ is real, then the vectors $g_1,\ldots,g_r$ can be chosen in $\Reals^n$.
Then $V$ and $X^{1/2}$ are real matrices and $f_1,\ldots,f_n$ are in $\Reals^r$.
\end{proof}

\begin{lemma}\label{lem:XtY}
Let $X\in M_n(\Cpx)$ be a positive semidefinite matrix and let $P$
be the support projection of $X$.
Then a Hermitian $n\times n$ matrix $Y$ has the property that there is $\eps>0$
such that $X+tY$ is positive semidefinite for all $t\in(-\eps,\eps)$ if and only if
$Y=PYP$.
\end{lemma}
\begin{proof}
If $X=0$ then this is trivially true, so suppose $X\ne0$.
After conjugating with a unitary, we may without loss of generality assume
$P=\diag(1,\ldots,1,0,\ldots,0)$ with $\rank(X)=\rank(P)=r$.
Then $PXP$, thought of as an $r\times r$ matrix, is positive definite.
By continuity of the determinant, we see that if $Y=PYP$, then $Y$ enjoys the property
described above.

Conversely, if $Y\ne PYP$, then we may choose two standard basis vectors $e_i$ and $e_j$ for $i\le r<j$,
such that the compressions of $X$ and $Y$ to the subspace spanned by $e_i$ and $e_j$ are given by the matrices
\[
\Xh=\left(\begin{matrix}
x&0\\0&0
\end{matrix}\right),\qquad
\Yh=\left(\begin{matrix}
a&b\\\overline{b}&c
\end{matrix}\right)
\]
for some $x>0$, $a,c\in\Reals$ and $b\in\Cpx$ with $c$ and $b$ not both zero.
But
\[
\det(\Xh+t\Yh)=txc+t^2(ac-|b|^2).
\]
If $c\ne0$, then $\det(\Xh+t\Yh)<0$ for all nonzero $t$ sufficiently small in magnitude and of the appropriate sign,
while if  $c=0$ then $b\ne0$ and $\det(\Xh+t\Yh)<0$ for all $t\ne0$.
\end{proof}

\begin{prop}\label{prop:rankX}
Let $n\in\Nats$, let $X\in\Theta_n$ and let $P$ be the support projection of $X$.
A necessary and sufficient condition for $X$ to be an extreme point
of $\Theta_n$ is that there be no nonzero Hermitian $n\times n$ matrix $Y$ having
zero diagonal and satisfying $Y=PYP$.
Consequently, if $X$ is an extreme point of $\Theta_n$, then $\rank(X)\le\sqrt n$.
\end{prop}
\begin{proof}
$X$ is an extreme point of $\Theta_n$ if and only if there is no nonzero Hermitian $n\times n$
matrix $Y$ such that $X+tY\in\Theta_n$ for all $t\in\Reals$ sufficiently small in magnitude.
Now use Lemma~\ref{lem:XtY} and the fact that $\Theta_n$ consists of the positive semidefinite matrices
with all diagonal values equal to $1$.

For the final statement, if $r=\rank(X)$ then the set of Hermitian matrices with support projection
under $P$ is a real vector space of dimension $r^2$, while the space of $n\times n$
Hermitian matrices with zero diagonal has dimension $n^2-n$.
If $r^2>n$, then the intersection of these two spaces is nonzero.
\end{proof}

\begin{prop}\label{prop:extrframe}
Let $X\in\Theta_n$.
Suppose $f_1,\ldots,f_n$ is a frame consisting of $n$ unit vectors in $\Cpx^r$, where $r=\rank(X)$,
so that $X=F^*F$ with $F=(f_1,\ldots,f_n)$ is the corresponding frame operator.
(See Lemma~\ref{lem:frame}.)
Then $X$ is an extreme point of $\Theta_n$ if and only if
the only $r\times r$ self--adjoint matrix $Z$ satisfying $\langle Zf_j,f_j\rangle=0$ for all $j\in\{1,\ldots,n\}$
is the zero matrix.
\end{prop}
\begin{proof}
Since $F$ is an $r\times n$ matrix of rank $r$,
the map $M_r(\Cpx)_{s.a.}\to M_n(\Cpx)_{s.a.}$ given by $Z\mapsto F^*ZF$ is an injective
linear map onto $PM_n(\Cpx)_{s.a.}P$, where $P$ is the support projection of $X$.
If $Y=F^*ZF$, then $Y_{jj}=\langle Zf_j,f_j\rangle$.
Thus, the condition for $X$ to be extreme now follows from the characterization found in
Proposition~\ref{prop:rankX}.
\end{proof}

\begin{prop}\label{prop:rank1}
Let $n\in\Nats$
and suppose $X\in\Theta_n$ satifies $\rank(X)=1$.
Then $X$ is an extreme point of $\Theta_n$
and $X\in\Fc_n$.
Moreover, using the notation introduced in Remark~\ref{rem:GNS}, we have
\begin{multline}\label{eq:cconvrank1}
\conv\{X\in\Theta_n\mid\rank(X)=1\}= \\
=\{\big(\tau(U_i^*U_j)\big)_{1\le i,j\le n}\mid\tau:\Afr\to\Cpx\text{ a positive trace, }
 \tau(1)=1,\,\pi_\tau(\Afr)\text{ commutative}\}
\end{multline}
and this set is closed in $\Theta_n$.
\end{prop}

\begin{notation}
We let $\Cc_n$ denote the set given in~\eqref{eq:cconvrank1}.
Thus, we have $\Cc_n\subseteq\Fc_n$.
Moreover, ({\em cf}\/ Remark~\ref{rem:GNS}), $\Cc_n$ is the set of matrices as in~\eqref{eq:X}
where $(U_1,\ldots,U_n)$ run over all $n$--tuples of commuting unitarires in C$^*$--algebras $A$ with faithful tracial
state $\tau$.
\end{notation}

\begin{proof}[Proof of Proposition~\ref{prop:rank1}]
By Lemma~\ref{lem:frame}, we have $X=F^*F$ where $F=(f_1,\ldots,f_n)$
for complex numbers $f_j$ with $|f_j|=1$.
Using Proposition~\ref{prop:extrframe}, we see immediately that $X$ is an extreme point of $\Theta_n$.
Thinking of each $f_j$ as a $1\times1$ unitary, we have $X\in\Fc_n$
and, moreover, $X=\big(\tau(U_i^*U_j)\big)_{1\le i,j\le n}$, where $\tau:\Afr\to\Cpx$
is the character defined by $\tau(U_i)=f_i$;
in fact, it is apparent that every character on $\Afr$ yields a rank one element of $\Theta_n$.
Since the set of traces $\tau$ on $\Afr$ having $\pi_\tau(\Afr)$ commutative is convex,
this implies the inclusion $\subseteq$ in~\eqref{eq:cconvrank1}.

That the left--hand--side of~\eqref{eq:cconvrank1} is compact follows from Caratheodory's theoem,
because the rank one projections form a compact set.
If $\tau:\Afr\to\Cpx$ is a positive trace with $\tau(1)=1$ and $\pi_\tau(\Afr)$ commutative,
then $\tau=\psi\circ\pi_\tau$ for a state $\psi$ on the C$^*$--algebra completion of
$\pi_\tau(\Afr)$.
Since every state on a unital, commutative C$^*$--algebra is in the closed convex hull of
the characters of that C$^*$--algebra, $\tau$ is itself the limit in norm of a sequence of
finite convex combinations of characters of $\Afr$.
Thus, $X=\big(\tau(U_i^*U_j)\big)_{1\le i,j\le n}$ is the limit of a sequence of finite convex
combinations of rank one elements of $\Theta_n$,
and we have $\supseteq$ in~\eqref{eq:cconvrank1}.
\end{proof}

\begin{remark}\label{rem:Fcproperties}
We see immediately from~\eqref{eq:cconvrank1} that $\Cc_n$ is a closed convex set that is closed
under conjugation with diagonal unitary matrices and permutation matrices;
also, since the set of rank one elements of $\Theta_n$ is closed under taking Schur products,
so is the set $\Cc_n$.
Furthermore, since $\Cc_n$ lies in a vector space of real dimension $m:=n^2-n$, by Caratheodory's theorem
every element of $\Cc_n$ is a convex combination of not more than $m+1$ rank one elements of $\Theta_n$.
\end{remark}

An immediate application of Propositions~\ref{prop:rankX} and~\ref{prop:rank1} is the following.
\begin{cor}\label{cor:Theta3}
The extreme points of $\Theta_3$ are precisely the rank one elements of $\Theta_3$.
Moreover, we have
\[
\Cc_3=\Fc_3=\Gc_3=\Theta_3.
\]
\end{cor}

\begin{remark}\label{rem:Z=0}
Let $X\in\Gc_n$ and take $A$, $\tau$ and $U_1,\ldots,U_n$ as in Definition~\ref{def:G}
so that~\eqref{eq:X} holds, and assume without loss of generality that $\tau$ is faithful on $A$.
If we identify $M_n(A)$ with $A\otimes M_n(\Cpx)$, then we have
$X=n(\tau\otimes\id_{M_n(\Cpx)})(P)$, where $P$ is the projection
\[
P=\frac1n
\left(\begin{matrix}U_1^* \\ U_2^* \\ \vdots \\ U_n^*\end{matrix}\right)
(U_1\;U_2\;\cdots\;U_n)
\]
in $M_n(A)$.
If
$c=(c_1,\ldots,c_n)^t\in\Cpx^n$ is
such that $Xc=0$,
then this yields $\tau(Z^*Z)=0$, where $Z=c_1U_1+\cdots+c_nU_n$.
Since $\tau$ is a faithful, we have $Z=0$.
\end{remark}

\begin{prop}\label{prop:rank2}
Let $n\in\Nats$.
If $X\in\Gc_n$ and $\rank(X)\le2$, then $X\in\Cc_n$.
\end{prop}
\begin{proof}
If $\rank(X)=1$, then this follows from Propostion~\ref{prop:rank1},
so assume $\rank(X)=2$.
Let $\tau:\Afr\to\Cpx$ be a positive, unital trace such that
$X=\big(\tau(U_i^*U_j)\big)_{1\le i,j\le n}$
and let $\pi_\tau:\Afr\to B(L^2(\Afr,\tau))$ the the $*$--representation
as described in Remark~\ref{rem:GNS}.
Let $\sigma:\Afr\to\pi_\tau(\Afr)$ be the $*$--representation defined by
$\sigma(U_i)=\pi_\tau(U_1)^*\pi_\tau(U_i)$ for each $i\in\{1,\ldots,n\}$ and let
$\tau'=\tau\circ\sigma$.
Then $\tau'$ is a positive, unital trace on $\Afr$ and the matrix
$\big(\tau'(U_i^*U_j)\big)_{1\le i,j\le n}$ is equal to $X$.
Furthermore, $\pi_{\tau'}(U_1)=I$.
Consequently, we may without loss of generality assume $\pi_\tau(U_1)=I$.

Let $e_1,\ldots,e_n$ denote the standard basis vectors of $\Cpx^n$.
Let $i,j\in\{2,\ldots,n\}$, with $i\ne j$.
Since $\rank(X)=2$, there are $c_1,c_i,c_j\in\Cpx$ with $c_1\ne0$ such that
$X(c_1e_1+c_ie_i+c_je_j)=0$.
By Remark~\ref{rem:Z=0}, we have $\pi_\tau(c_1I+c_iU_i+c_jU_j)=0$.
We do not have $c_i=c_j=0$, so assume $c_i\ne0$.
If $c_j=0$, then $\pi_\tau(U_i)$ is a scalar multiple of the identity,
while if $c_j\ne0$, then $\pi_\tau(U_i)$ and $\pi_\tau(U_j)$ generate the same C$^*$--algebra,
which is commutative.
In either case, we have that the $*$--algebras generated by $\pi_\tau(U_i)$ and $\pi_\tau(U_j)$
commute with each other.
Therefore, $\pi_\tau(\Afr)$ is commutative, and $X\in\Cc_n$.
\end{proof}

\begin{cor}\label{cor:Theta4}
$\Gc_4\ne\Theta_4$.
\end{cor}
\begin{proof}
Combining Proposition~\ref{prop:rank2} and Proposition~\ref{prop:rank1},
we see that $\Gc_4$ has no extreme points of rank $2$.
It will suffice to find an extreme point $X$ of $\Theta_4$ with
$\rank(X)=2$.
By Proposition~\ref{prop:extrframe}, it will suffice to find four unit vectors
$f_1,\ldots,f_4$ spanning $\Cpx^2$ such that the only self--adjoint
$Z\in M_2(\Cpx)$
satisfying $\langle Zf_i,f_i\rangle=0$ for all $i=1,\ldots,4$ is the
zero matrix.
It is easily verified that the frame
\[
f_1=\begin{pmatrix}1\\0\end{pmatrix},\quad
f_2=\begin{pmatrix}0\\1\end{pmatrix},\quad
f_3=\begin{pmatrix}1/\sqrt2\\1/\sqrt2\end{pmatrix},\quad
f_4=\begin{pmatrix}i/\sqrt2\\1/\sqrt2\end{pmatrix}
\]
does the job, and, with $F=(f_1,f_2,f_3,f_4)$, this yields the matrix
\begin{equation}\label{eq:X4}
X=F^*F=
\left(\begin{matrix}
1&0&\frac1{\sqrt2}&\frac{i}{\sqrt2} \\[1ex]
0&1&\frac1{\sqrt2}&\frac1{\sqrt2} \\[1ex]
\frac1{\sqrt2}&\frac1{\sqrt2}&1&\frac{1+i}2 \\[1ex]
\frac{-i}{\sqrt2}&\frac1{\sqrt2}&\frac{1-i}2&1
\end{matrix}\right)\in\Theta_4\backslash\Gc_4\,.
\end{equation}
\end{proof}

\begin{remark}
We cannot have $\Cc_n=\Fc_n$ for all $n$, because by
an easy a modification of Kirchberg's proof of Proposition~4.6 of~\cite{Ki93},
this would imply that $M_2(\Cpx)$ can be faithfully represented in a commutative von Neumann algebra.
(This argument shows that for some $n$ there must be two--by--two unitaries $V_1,\ldots,V_n$ such that
the matrix $\big(\tr_2(V_i^*V_j)\big)_{1\le i,j\le n}$ does not belong to $\Cc_n$.)
In fact, in Proposition~\ref{prop:F6} we will show $\Fc_6\ne\Cc_6$.
However, we don't know whether $\Fc_n=\Cc_n$ holds or not for $n=4$ or $n=5$.
\end{remark}

\section{Real matrices}
\label{sec:Reals}

The main result of this section is the following, which easily follows from the usual representation
of the Clifford algebra.

\begin{thm}\label{thm:Reals}
For every $n\in\Nats$, we have
\[
M_n(\Reals)\cap\Theta_n\subseteq\Fc_n\,.
\]
\end{thm}

We first recall the representation of the Clifford algebra.
Let $\Lambda$ be a linear map from a real Hilbert space $H$ into
the bounded, self--adjoint operators $B(\KEu)_{s.a.}$, for some complex Hilbert space $\KEu$,
satisfying
\begin{equation}\label{eq:Lambdaxy}
\Lambda(x)\Lambda(y)+\Lambda(y)\Lambda(x)=2\langle x,y\rangle I_{H},\qquad(x,y\in H).
\end{equation}
The real algebra generated by range of $\Lambda$ is uniquely determined by $H$ and called
the real Clifford algebra.\\

Consider a real Hilbert space $H$ of finite dimension $r$ with its canonical basis $\{e_i\}$. Let
\begin{gather*}
U=\left( \begin{array}{cc}
1 & 0 \\
0 & -1 \\
\end{array} \right),
\quad
V=\left( \begin{array}{cc}
0 & 1 \\
1 & 0 \\
\end{array} \right),
\quad
I_2=\left( \begin{array}{cc}
1 & 0 \\
0 & 1 \\
 \end{array} \right).
\end{gather*}
Then the real Clifford algebra of $H$ has the following representation by $2^r\times2^r$ matrixes
\[
\Lambda(x)=\sum \lambda_i U^{\otimes i-1}\otimes V\otimes I_2^{\otimes (n-i)},
\]
where $x=\sum \lambda_i e_i$.
It easy to check that the relation \eqref{eq:Lambdaxy} is satisfied.
Moreover if $||x||=1$ then $\Lambda(x)$ is symmetry, i.e. $\Lambda(x)^* = \Lambda(x)$ and $\Lambda(x)^2=I$.

\begin{proof}[Proof of Theorem~\ref{thm:Reals}.]
Let $r$ be the rank of $X$.
By Lemma~\ref{lem:frame}, there are unit vectors $f_1,\ldots,f_n\in\Reals^r$ such that $X_{i,j}=\langle f_i,f_j\rangle$
for all $i$ and $j$.
Taking $\Lambda$ as described above, we get $2^r\times 2^r$ unitary matrices $\Lambda(f_i)$ (in fact, they are symmetries),
and from~\eqref{eq:Lambdaxy} we have $\tr(\Lambda(f_i)\Lambda(f_j))=\langle f_i,f_j\rangle$.
\end{proof}

Below is the result for real matrices that is entirely analogous to Proposition~\ref{prop:rankX}.

\begin{prop}\label{prop:rankRealX}
Let $n\in\Nats$, let $X\in M_n(\Reals)\cap\Theta_n$ and let $P$ be the support projection of $X$.
A necessary and sufficient condition for $X$ to be an extreme point
of $ M_n(\Reals)\cap\Theta_n$ is that there be no nonzero Hermitian real $n\times n$ matrix $Y$ having
zero diagonal and satisfying $Y=PYP$.
Consequently, if $X$ is an extreme point of $M_n(\Reals)\cap\Theta_n$ and $r=\rank(X)$,
then $r(r+1)/2\le n$.
\end{prop}
\begin{proof}
This is just like the proof of Proposition~\ref{prop:rankX},
the only difference being that the dimension of $P M_n(\Reals)_{s.a.}P$
for a projection $P$ of rank $r$ is $r(r+1)/2$.
\end{proof}

\begin{cor}\label{cor:n<6}
If $n\le5$, then
\begin{equation}\label{eq:MnRFc}
M_n(\Reals)\cap\Theta_n\subseteq\Cc_n.
\end{equation}
\end{cor}
\begin{proof}
From Proposition~\ref{prop:rankRealX}, we see that every extreme point $X$ of $M_n(\Reals)\cap\Theta_n$ for $n\le5$
has rank $r\le2$.
But $X\in\Fc_n\subseteq\Gc_n$, by Theorem~\ref{thm:Reals},
so using Proposition~\ref{prop:rank2}, it follows that all extreme points of $M_n(\Reals)\cap\Theta_n$
lie in $\Cc_n$.
Since $\Cc_n$ is closed and convex (see Proposition~\ref{prop:rank1}), the inclusion~\eqref{eq:MnRFc} follows.
\end{proof}

Of course, we also have the result for real matrices (and real frames) that is analogous to Proposition~\ref{prop:extrframe},
which is stated below.
The proof is the same.
\begin{prop}\label{prop:extrRealframe}
Let $X\in M_n(\Reals)\cap\Theta_n$.
Suppose $f_1,\ldots,f_n$ is a frame consisting of $n$ unit vectors in $\Reals^r$, where $r=\rank(X)$,
so that $X=F^*F$ with $F=(f_1,\ldots,f_n)$ is the corresponding frame operator.
(See Lemma~\ref{lem:frame}.)
Then $X$ is an extreme point of $M_n(\Reals)\cap\Theta_n$ if and only if
the only real Hermitian $r\times r$ matrix $Z$ satisfying $\langle Zf_j,f_j\rangle=0$ for all $j\in\{1,\ldots,n\}$
is the zero matrix.
\end{prop}

Although Corollary~\ref{cor:n<6} shows that every element of $M_n(\Reals)\cap\Theta_n$ for $n\le5$
is in the closed convex hull of the rank one operators in $\Theta_n$, it is not true that every
element of $M_n(\Reals)\cap\Theta_n$ lies in the closed convex hull of rank one operators in $M_n(\Reals)\cap\Theta_n$,
even for $n=3$,
as the following example shows.

\begin{example}
Consider the frame
\[
f_1=\begin{pmatrix}1\\0\end{pmatrix},\quad
f_2=\begin{pmatrix}0\\1\end{pmatrix},\quad
f_3= \frac1{\sqrt2}\begin{pmatrix}1\\1\end{pmatrix}
\]
of three unit vectors in $\Reals^2$.
It is easily verified that the only real Hermitian $2\times 2$ matrix $Z$ such that $\langle Zf_i,f_i\rangle=0$
for all $i=1,2,3$ is the zero matrix.
Thus, by Proposition~\ref{prop:extrRealframe},
\[
X=\left(\begin{matrix}
1&0&\frac1{\sqrt2} \\[1ex]
0&1&\frac1{\sqrt2} \\[1ex]
\frac1{\sqrt2}&\frac1{\sqrt2}&1
\end{matrix}\right)
\]
is a rank--two extreme point of $M_3(\Reals)\cap\Theta_3$.
However, an explicit decomposition as a convex combination of rank one operators in $\Theta_3$ is
\[
X=\frac12\left(\begin{matrix}
1&i&\frac{1+i}{\sqrt2} \\[1ex]
-i&1&\frac{1-i}{\sqrt2} \\[1ex]
\frac{1-i}{\sqrt2}&\frac{1+i}{\sqrt2}&1
\end{matrix}\right)
+\frac12\left(\begin{matrix}
1&-i&\frac{1-i}{\sqrt2} \\[1ex]
i&1&\frac{1+i}{\sqrt2} \\[1ex]
\frac{1+i}{\sqrt2}&\frac{1-i}{\sqrt2}&1
\end{matrix}\right).
\]
\end{example}

\begin{prop}\label{prop:F6}
We have
\[
M_6(\Reals)\cap\Theta_6\not\subseteq\Cc_6\,.
\]
Thus, we have $\Fc_6\ne\Cc_6$.
\end{prop}
\begin{proof}
We construct an example of $X\in(M_6(\Reals)\cap\Theta_6)\backslash\Cc_6$.
In fact, it will be a rank--three extreme point of $M_6(\Reals)\cap\Theta_6$.

Consider the frame
\begin{alignat*}{3}
f_1&=\begin{pmatrix}1\\0\\0\end{pmatrix},\quad &
f_2&=\begin{pmatrix}0\\1\\0\end{pmatrix},\quad &
f_3&=\begin{pmatrix}0\\0\\1\end{pmatrix},\quad \\
f_4&= \frac1{\sqrt2}\begin{pmatrix}1\\1\\0\end{pmatrix},\quad &
f_5&= \frac1{\sqrt2}\begin{pmatrix}0\\1\\1\end{pmatrix},\quad &
f_6&= \frac1{\sqrt3}\begin{pmatrix}1\\1\\1\end{pmatrix}
\end{alignat*}
of six unit vectors in $\Reals^3$.
It is easily verified that the only real Hermitian $3\times 3$ matrix $Z$ such that $\langle Zf_i,f_i\rangle=0$
for all $i\in\{1,\ldots,6\}$ is the zero matrix.
Thus, by Proposition~\ref{prop:extrRealframe},
\[
X=\left(\begin{matrix}
1 & 0 & 0 & \frac{1}{\sqrt{2}} & 0 & \frac{1}{\sqrt{3}} \\[1ex]
0 & 1 & 0 & \frac{1}{\sqrt{2}} & \frac{1}{\sqrt{2}} & \frac{1}{\sqrt{3}} \\[1ex]
0 & 0 & 1 & 0 & \frac{1}{\sqrt{2}} & \frac{1}{\sqrt{3}} \\[1ex]
\frac{1}{\sqrt{2}} & \frac{1}{\sqrt{2}} & 0 & 1 & \frac{1}{2} & \sqrt{\frac{2}{3}} \\[1ex]
0 & \frac{1}{\sqrt{2}} & \frac{1}{\sqrt{2}} & \frac{1}{2} & 1 & \sqrt{\frac{2}{3}} \\[1ex]
\frac{1}{\sqrt{3}} & \frac{1}{\sqrt{3}} & \frac{1}{\sqrt{3}} & \sqrt{\frac{2}{3}} & \sqrt{\frac{2}{3}}  & 1
\end{matrix}\right)
\]
is a rank--three extreme point of $M_6(\Reals)\cap\Theta_6$.
The nullspace of $X$ is spanned by the vectors
\begin{align*}
v_1&=\textstyle (\frac1{\sqrt2},\frac1{\sqrt2},0,-1,0,0)^t \\
v_2&=\textstyle (0,\frac1{\sqrt2},\frac1{\sqrt2},0,-1,0)^t \\
v_3&=\textstyle (\frac1{\sqrt3},\frac1{\sqrt3},\frac1{\sqrt3},0,0,-1)^t.
\end{align*}

Suppose, to obtain a contradiction, that we have $X\in\Cc_6$.
Then there is a commutative C$^*$--algebra $A=C(\Omega)$ with a faithful tracial state $\tau$
and there are unitaries $I=U_1,U_2,\ldots,U_6\in A$ such that $X=\big(\tau(U_i^*U_j)\big)_{1\le i,j\le 6}$.
Taking the vectors $v_1$, $v_2$ and $v_3$, above,
by Remark~\ref{rem:Z=0} we have
\begin{align}
U_4&=\frac1{\sqrt2}(U_1+U_2) \label{eq:U4} \\
U_5&=\frac1{\sqrt2}(U_2+U_3) \label{eq:U5} \\
U_6&=\frac1{\sqrt3}(U_1+U_2+U_3). \label{eq:U6}
\end{align}
Fixing any $\omega\in\Omega$, we have that $\zeta_j:=U_j(\omega)$ is a point on the unit circle $\Tcirc$, ($1\le j\le 6$).
From~\eqref{eq:U4} and $|\zeta_4|=1$, we get $\zeta_1=\pm i\zeta_2$ and  similarly from~\eqref{eq:U5} we get $\zeta_3=\pm i\zeta_2$.
However, from~\eqref{eq:U6}, we then have
\[
\zeta_6\in\{\frac{1-2i}{\sqrt3}\zeta_2,\frac1{\sqrt3}\zeta_2,\frac{1+2i}{\sqrt3}\zeta_2\},
\]
which contradicts $|\zeta_6|=|\zeta_2|=1$.
\end{proof}

\section{Nonempty interior}
\label{sec:int}

In this section, we show that the interior of $\Fc_n$ and, in fact, of $\Cc_n$, is nonempty,
when considered as a subset of $\Theta_n$.
(Since $\Cc_n=\Theta_n$ for $n=1,2,3$, this needs proving only for $n\ge4$.)

Given $X\in\Theta_n$, let 
\begin{align*}
a_X&=\sup\{t\in[0,1]\mid tX+(1-t)I\in\Fc_n\} \\
c_X&=\sup\{t\in[0,1]\mid tX+(1-t)I\in\Cc_n\}.
\end{align*}
Of course, $c_X\le a_X$.
We now show that $c_X$ is bounded below by a nonzero constant that depends only on $n$.
In particular, we have that the identity element lies in the interior of $\Cc_n$, when this is taken
as a subset of the affine space of self--adjoint matrices having all diagonal entries equal to $1$.

\begin{prop}\label{prop:int}
Let $n\in\Nats$, $n\ge3$, and let $X\in \Theta_{n}$.
Then
\begin{equation}\label{eq:CX}
c_X\ge\frac6{n^2-n}\,.
\end{equation}
Moreover, if
$\lambda_0$ is the smallest eigenvalue of $X$, then
\begin{equation}\label{eq:CXlambda}
c_X\ge\min(\frac6{(n^2-n)(1-\lambda_0)},1).
\end{equation}
\end{prop}
\begin{proof}
We have $X=(x_{ij})_{i,j=1}^n$ with $x_{ii}=1$ for all $i=1,\ldots, n$.
Denote $G=\{ \sigma \in S_n \mid \sigma(1)< \sigma(2) <\sigma(3)\}$.
Then
\[
\# G = \binom n 3 (n-3)!
\]
Let $U_\sigma = (u_{ij})$ be the
permutation unitary matrix where $u_{ij} = \delta_{i, \sigma(i)}$.
Then $U^* X U = (x_{\sigma^{-1}(i) \sigma^{-1}(j)})_{i,j}$. Define
the block-diagonal matrix 
\[
B_\sigma = \left(\begin{array}{ccc}
1 & x_{\sigma(1)\sigma(2)} & x_{\sigma(1)\sigma(3)}\\
x_{\sigma(2)\sigma(1)} & 1 & x_{\sigma(2)\sigma(3)}\\
x_{\sigma(3)\sigma(1)} & x_{\sigma(3)\sigma(2)} & 1\\
\end{array}\right)\oplus I_{n-3}\,.
\]
Using Corollary~\ref{cor:Theta3} (and Remark~\ref{rem:Fcproperties}), we easily see $B_\sigma\in\Cc_n$.

Let $J_\sigma =\{ (\sigma(1), \sigma(2)), (\sigma(1), \sigma(3)),
(\sigma(2), \sigma(3))\}$.
Put  $X_\sigma = U^* B_\sigma U$. Then
\[
(X_\sigma)_{k\ell} = \begin{cases} 0, & \text{ if } (k,\ell)\not\in  \{(1,1), \ldots, (n,n)\}\cup J_\sigma, \\
1, & \text{ if }k=\ell \\
x_{k\ell}, & \text{ if } (k,\ell)\in J_\sigma\,.
\end{cases}
\]

Since for any $k<\ell$ we have
\begin{gather*}
 \# \{\sigma \in G \mid \sigma(1) =k, \sigma(2)=\ell \text{ or } \sigma(1) =k, \sigma(3)=\ell  \text{ or } \sigma(2) =k, \sigma(3)=\ell \} = \\
 ((n-\ell) + (\ell-k-1) +(k-1)) (n-3)! = (n-2)!
\end{gather*}
it follows that matrix
\[
X'= \frac{1}{\# G} \sum_{\sigma \in G} X_\sigma
\]
has entries
$x'_{ii}=1$, and $x'_{k\ell} =\frac{6}{n^2-n} x_{k\ell}$ if $k\ne \ell$.

Since $\Cc_n$ is closed under conjugating with permutation matrices, we have
$X_\sigma \in \Cc_n$ for all $\sigma\in G$.
But then the average
$X'$ also belongs to $\Cc_n$.
This implies~\eqref{eq:CX}.

Now~\eqref{eq:CXlambda} is an easy consequence of~\eqref{eq:CX}.
Indeed, if $\lambda_0=1$, then $X$ is the identity matrix and $c_X=1$.
If $\lambda_0<1$, then let $Y=\frac1{1-\lambda_0}(X-\lambda_0I)$.
We have $Y\in\Theta_n$, and
\[
(1-t)I+tY=(1-\frac t{1-\lambda_0})I+\frac t{1-\lambda_0}X.
\]
This implies $c_X\ge\min(1,\frac{c_Y}{1-\lambda_0})$.
\end{proof}

Given an $n\times n$ matrix $A=(a_{ij})_{1\le i,j\le n}$, let $\Abar$ denote matrix whose $(i,j)$ entry is the the complex conjugate of $a_{ij}$.
If $A$ is self--adjoint, then so is $\Abar$, and these two matrices have the same eigenvalues (and multiplicities).
Consequently, $A-\Abar$ has spectrum that is symmetric about zero.

\begin{lemma}\label{lem:dcx}
Let $X\in\Theta_n$ and let $d>0$ be such that
\[
I+d\left(\frac{X-\Xbar}2\right)\in\Fc_n\,.
\]
Then $a_X\ge d/(d+1)$.
If $n\le 5$ and 
\begin{equation}\label{eq:inFc}
I+d\left(\frac{X-\Xbar}2\right)\in\Cc_n\,,
\end{equation}
then $c_X\ge d/(d+1)$.
\end{lemma}
\begin{proof}
The matrix $(X+\Xbar)/2$ is real and lies in $\Theta_n$.
Using Theorem~\ref{thm:Reals}, we have $(X+\Xbar)/2\in\Fc_n$.
Thus, we have
\[
\frac1{d+1}I+\frac d{d+1}X=\frac1{d+1}\left(I+d\left(\frac{X-\Xbar}2\right)\right)+\frac d{d+1}\left(\frac{X+\Xbar}2\right)
\in\Fc_n\,.
\]
If $n\le5$ and~\eqref{eq:inFc} holds, then we similarly apply Corollary~\ref{cor:n<6}.
\end{proof}

\begin{example}
Consider the matrix $X$ as in~\eqref{eq:X4}, from Corollary~\ref{cor:Theta4}.
From Proposition~\ref{prop:int} and closedness of $\Fc_n$, we know
$\frac12\le c_X\le a_X<1$.
It would be interesting to know the precise value of $a_X$, in order to have a concrete example
of an element on the boundary of $\Fc_4$ in $\Theta_4$.

Since
\[
\frac{X-\Xbar}2=
\begin{pmatrix}
0&0&0&\frac i{\sqrt2} \\
0&0&0&0 \\
0&0&0&\frac i2 \\
-\frac i{\sqrt2}&0&-\frac i2&0
\end{pmatrix}
\]
has norm $\sqrt3/2$ and since it is conjugate by a permutation matrix to an element of $M_3(\Cpx)\oplus\Cpx$,
using Corollary~\ref{cor:Theta3} we have that~\eqref{eq:inFc} holds with $d=2/\sqrt3$.
A slightly better value is obtained by letting $Y$ be the result of conjugation
of $X$ with the diagonal unitary $\diag(1,1,1,e^{-i\pi/4})$.
Then
\[
\frac{Y-\Ybar}2=
\begin{pmatrix}
0&0&0&\frac i2 \\
0&0&0&-\frac i2 \\
0&0&0&0 \\
-\frac i2&\frac i2&0&0
\end{pmatrix}
\]
which has norm $1/\sqrt2$ and similarly yields $d=\sqrt2$.
Applying Lemma~\ref{lem:dcx} gives $c_X=c_Y\ge\sqrt2/(1+\sqrt2)\approx0.586$.
\end{example}

\medskip
\noindent{\bf Acknowledgment.}
The authors thank Vern Paulsen for kindly directing them to the literature on extreme
correlation matrices.

\bibliographystyle{plain}

\end{document}